\documentclass{amsart}
\usepackage{amsmath}
\usepackage{amsthm}
\usepackage{amssymb}
\usepackage[german,english]{babel}
\usepackage[utf8]{inputenc}
\usepackage{microtype}
\usepackage[T1]{fontenc}
\usepackage{epsfig,fancybox,color}
\usepackage{enumerate}
\usepackage{setspace}
\usepackage{mathtools}  
\usepackage[linesnumbered,ruled]{algorithm2e}
\usepackage{todonotes}
\usepackage{tikz}
\usepackage{tikz-cd} 
\usetikzlibrary{arrows.meta}
\usepackage{hyperref}
\definecolor{darkred}{RGB}{160,0,0}
\definecolor{darkblue}{RGB}{0,0,160}
\hypersetup{
  colorlinks,
  citecolor=darkblue,
  filecolor=black,
  linkcolor=darkblue,
  urlcolor=darkblue
}

\usepackage[
  text={414pt,575pt},
  headheight=9pt,
  centering
]{geometry}

\newcommand{\excise}[1]{}

\theoremstyle{definition}
\newtheorem{thm}{Theorem}
\newtheorem*{thm*}{Theorem}
\newtheorem{lem}[thm]{Lemma}

\newtheorem{prop}[thm]{Proposition}
\newtheorem{conj}[thm]{Conjecture}

\newtheorem{example}[thm]{Example}
\newtheorem{remark}[thm]{Remark}

\numberwithin{equation}{section}

\newcommand{\ring}[1]{\ensuremath{\mathbb{#1}}}

\renewcommand{\>}{\rangle}
\newcommand{\<}{\langle}

\newcommand{\NN}{\ring{N}}
\newcommand{\QQ}{\ring{Q}}

\newcommand\defas{\coloneqq}

\newcommand\set[1]{\{#1\}}
\newcommand\abs[1]{|#1|}

\DeclareMathOperator{\maps}{\rightarrow} 
\DeclareMathOperator{\supp}{supp} 
\DeclareMathOperator{\Span}{span} 
\DeclareMathOperator{\sgn}{sgn} 

\begin{document}

\selectlanguage{english}

\title{Short polynomials in determinantal ideals}

\author{Thomas Kahle and Finn Wiersig}


\date{\today}

\keywords{determinantal ideal, sparse polynomial}

\thanks{The authors were supported by the German Research Foundation DFG --
  314838170, GRK~2297 MathCoRe.}

\makeatletter
\@namedef{subjclassname@2020}{\textup{2020} Mathematics Subject Classification}
\makeatother
\subjclass[2020]{14M12, 14Q20, 15A15}


\begin{abstract}
  We show that a determinantal ideal generated by $t$-minors does not
  contain any nonzero polynomials with $t!/2$ or fewer terms.
  Geometrically this means that any nonzero polynomial vanishing on
  all matrices of rank at most $t-1$ has more than $t!/2$ terms.
\end{abstract}

\maketitle



\section{Introduction}

In many areas of computational mathematics sparsity is an essential
feature used for complexity reduction.  Sparse mathematical objects
often allow more compact data structures and more efficient
algorithms.  We are interested in sparsity, that is having few terms,
as a complexity measure for polynomials, augmenting the usual degree
based complexity measures such as the Castelnuovo--Mumford regularity.

Sparsity based complexity applies to geometric objects too.  If
$X\subset K^{n}$ is a subset of affine $K$-space, one can ask for the
sparsest polynomial that vanishes on~$X$.  A monomial vanishes on $X$
if and only if $X$ is contained in the union of the coordinate
hyperplanes.  That $X$ is cut out by binomials can be characterized
geometrically using the log-linear geometry of binomial
varieties~\cite[Theorem~4.1]{eisenbud96:_binom_ideal}.  Algorithmic
tests for single binomials vanishing on $X$ are available both
symbolically~\cite{JKK17} and
numerically~\cite{hauenstein2021binomiality}.  We ask for the
\emph{shortest} polynomial vanishing on $X$, or algebraically, the
shortest polynomial in an ideal of the polynomial ring.  The shortest
polynomials contained in (principal) ideals of a univariate polynomial
ring have been considered in \cite{giesbrecht2010computing}.
Computing the shortest polynomials of an ideal in a polynomial ring
seems to be a hard problem with an arithmetic flavor.  Consider
Example~2 from \cite{JKK17}: For any $n\in\NN$, let
$I_n=\<(x-z)^2,nx-y-(n-1)z\>\subseteq\QQ[x,y,z]$. The ideals $I_n$ all
have Castelnouvo-Mumford regularity 2 and are primary over
$\<x-z,y-z\>$.  Then $I_{n}$ contains the binomial $x^n-yz^{n-1}$ and
there is no binomial of degree less than $n$ in~$I_{n}$.  This means
that the syzygies and also primary decomposition carry no information
about shortness.  It is unknown to the authors if a Turing machine can
decide if an ideal contains a polynomial with at most $t$ terms.

In this note we show that no short polynomials vanish on the set of
fixed rank matrices.
\begin{thm*}
  For $t\le m,n$, let $X_{t-1} \subset K^{m\times n}$ be the set of
  $m\times n$-matrices of rank at most~$t-1$ over a field~$K$.  There
  is no nonzero polynomial with $t!/2$ or fewer terms vanishing on all
  of~$X_{t-1}$.
\end{thm*}

In the rest of the introduction we fix notation.
Section~\ref{sec:short-and-linearforms} lays the foundations of our
approach.  We consider the space of coefficients in the monomial
basis.  Searching for a short polynomial is searching for many
simultaneously vanishing coefficients.  In
Section~\ref{sec:shortPoly-in-detIdeals} we specialize this to
determinantal ideals and prove the theorem.

\subsection*{Notation and conventions}

Let $K$ be a field and $R = K[x_1,\dots,x_n]$ the polynomial ring in
$n$ indeterminates with coefficients in~$K$.  The \emph{support} of
$f=\sum_{\alpha \in \NN^{n}}f_{\alpha}x^{\alpha}$ is
$\supp(f) = \set{\alpha\colon f_{\alpha}\neq 0}$.  A nonzero
polynomial $f\in R\setminus\set{0}$ is \emph{$t$-short} if it has at
most $t$ terms, that is, if $|\supp(f)| \le t$.  The \emph{shortness}
of an ideal $I\subset R$ is the minimal $t$ such that $I$ contains a
$t$-short polynomial.
We think of the shortness of an ideal as an important complexity
measure and aim to develop methods to determine it.

We write $\NN$ for the natural numbers without zero and
$\NN_0:=\NN\cup\set{0}$.  For any tuple $a=(a_1,\dots,a_n)\in\NN_0^n$,
let $|a|:=\sum_{i=1}^na_i$.  The set of exponents of monomials of
degree $d\in\NN_0$ in $K[x_1,\dots,x_n]$ is
$M_d^n:=\set{\alpha\in\NN_0^n | |\alpha|=d}$.  In a polynomial ring
$K[X]$, where $X$ denotes an $m\times n$-matrix of variables, the set
of exponents of monomials of degree $d\in\NN_0$ is $M_d^{m\times n}$.
Finally, for any $n\in\NN$, we abbreviate $[n]:=\set{1,\dots,n}$.


\section{Short polynomials and relations between linear forms}
\label{sec:short-and-linearforms}

To understand short polynomials in a homogenous ideal we study their
expressions in terms of fixed generators for the ideal.  A short
polynomial must produce many cancellations which we aim to detect
systematically.  To describe the general idea, let
$f_1,\dots,f_r\in K[x_1,\dots,x_n]$ be homogenous forms of degree~$t$.
In the monomial basis we write
$f_i=\sum_{\beta\in M_t^{n}}f_{i,\beta}x^{\beta}$ for all
$i=1,\dots,r$.  Fix a number $d\in\NN_0$.  We aim to compute the
smallest number of terms of a nonzero polynomial in $I^{(t+d)}$,
the degree $(t+d)$ homogeneous component of the ideal
$I=\<f_1,\dots,f_r\>$.

For every exponent $\alpha\in M_{t+d}^{n}$, consider the linear map
\begin{equation*}
\tilde{p}_{\alpha} \colon I^{(t+d)} \maps K,
\end{equation*}
sending a polynomial
$g=\sum_{\alpha\in M_{t+d}^{n}}g_{\alpha}x^{\alpha}\in I^{(t+d)}$ to
the coefficient $g_{\alpha}$ of~$x^{\alpha}$.  We would like to
understand how many of these maps $\tilde{p}_{\alpha}$ can vanish
simultaneously on one polynomial~$g$.  It is useful to pull this
information back to the coefficients as follows.  Consider the
surjective linear map
\begin{equation*}
  \epsilon^{(t+d)} \colon \bigoplus_{i=1}^{r}K[x_1,\dots,x_n]^{(d)} \maps I^{(t+d)}, \qquad
  (g_1,\dots,g_r) \mapsto \sum_{i=1}^{r}g_{i}f_{i}.
\end{equation*}
Now for every exponent $\alpha\in M_{t+d}^{n}$ let $p_{\alpha}$ be the
composition of $\epsilon^{(t+d)}$ and $\tilde{p}_{\alpha}$.  That is,
$p_{\alpha}$ is a linear form that makes the following diagram
commutative:
\begin{equation*}
\begin{tikzcd}
  \bigoplus_{i=1}^{r}K[x_1,\dots,x_n]^{(d)} \arrow{r}{\epsilon^{(t+d)}}\arrow{rd}{p_{\alpha}} & I^{(t+d)} \arrow{d}{\tilde{p}_{\alpha}} \\
                                                                                          & K
\end{tikzcd}.
\end{equation*}

The linear forms $p_{\alpha}$ extract the coefficients of a polynomial
in $I^{(t+d)}$ with respect to the chosen decomposition
$\sum_{i=1}^{r} f_{i}g_{i}$ in the monomial basis.  This can be
summarized as follows.
\begin{lem}\label{lem:p-alpha}
  The evaluation $p_{\alpha}(g_1,\dots,g_r)$ equals the coefficient of
  $x^{\alpha}$ for each polynomial expression
  $\sum_{i=1}^{r}g_{i}f_{i}\in I^{(t+d)}$.  In total, 
  \begin{equation*}
    I^{(t+d)} = \left\{\sum_{\alpha\in M_{t+d}^{n}}p_{\alpha}(g_1,\dots,g_r)x^{\alpha} \colon (g_1,\dots,g_r)\in\bigoplus_{i=1}^{r}K[x_1,\dots,x_n]^{(d)}\right\}.
  \end{equation*}
\end{lem}

For every $i=1,\dots,r$ let
$e_i\in\bigoplus_{i=1}^{r}K[x_1,\dots,x_n]$ denote the tuple that is
$1$ in the $i$-th entry and $0$ everywhere else.  Then
$\set{x^{\gamma}e_i : i=1,\dots,r, \, \gamma\in M_d^{n}}$ is a basis
of $\bigoplus_{i=1}^{r}K[x_1,\dots,x_n]^{(d)}$ as a vector space.  Let
\begin{equation*}
  \set{y_{i,\gamma} : i=1,\dots,r,\, \gamma\in M_d^{n}}
\end{equation*}
be the dual basis of
$\left(\bigoplus_{i=1}^{r}K[x_1,\dots,x_n]^{(d)}\right)^{*}$.  The
following lemma gives a concrete description of the $p_{\alpha}$ with
respect to this basis.  For this, generators $f_{1},\dots, f_{r}$ of
degree $t$ are fixed, as they need to be to just define
the~$p_{\alpha}$.

\begin{lem}\label{lem:palpha-formula}
  For every $\alpha\in M_{t+d}^{n}$ we have
  \begin{equation*}
    p_{\alpha}=\sum_{i=1}^{r}\sum_{\substack{\beta\in M_t^{n},\gamma\in M_d^{n}\\\beta+\gamma=\alpha}}f_{i,\beta}y_{i,\gamma}.
 \end{equation*}
\end{lem}

\begin{proof}
  Fix $(g_1,\dots,g_r)\in\bigoplus_{i=1}^{r}K[x_1,\dots,x_n]^{(d)}$,
  write $g_{i}=\sum_{\alpha\in M_d^{n}}g_{i,\alpha}x^{\alpha}$ in the monomial
  basis for all $i=1,\dots,r$, and calculate
\begin{align*}
  p_{\alpha}(g_1,\dots,g_r) &= \tilde{p}_{\alpha}\left(\epsilon^{(t+d)}(g_1,\dots,g_r)\right) \\
                          &= \tilde{p}_{\alpha}\left(\sum_{i=1}^{r}g_if_i\right) \\
                          &= \tilde{p}_{\alpha}\left(\sum_{i=1}^r\left(\sum_{\alpha\in M_d^{n}}g_{i,\alpha}x^{\alpha}\right)\left(\sum_{\beta\in M_t^{n}}f_{i,\beta}x^{\beta}\right)\right)\\
                          &= \tilde{p}_{\alpha}\left(\sum_{\alpha\in M_{t+d}^{n}}\sum_{i=1}^r\sum_{\substack{\beta\in M_t^{n},\gamma\in\NN_d^{n}\\\beta+\gamma=\alpha}} f_{i,\beta}g_{i,\gamma}x^{\alpha}\right) \\
                          &= \sum_{i=1}^r\sum_{\substack{\beta\in M_t^{n},\gamma\in\NN_d^{n}\\\beta+\gamma=\alpha}} f_{i,\beta}g_{i,\gamma}. \qedhere
\end{align*}
\end{proof}

\begin{example}\label{ex:monomialId}
  Consider the monomial ideal
  $\<x_{1},x_{2}\>\subset \QQ[x_{1},x_{2}]$.  The linear forms
  $p_{\alpha}$ depend only on the chosen generating set
  $f_{1} = x_{1}, f_{2} = x_{2}$ of the ideal.  For degree two (that
  is, $t=d=1$), these linear forms are contained in the 4-variate
  polynomial ring
  $K[y_{1,(1,0)}, y_{1,(0,1)}, y_{2,(1,0)}, y_{2,(0,1)}]$.  They are
  \begin{equation*}
    p_{(2,0)} = y_{1,(1,0)}, \qquad\quad
    p_{(1,1)} = y_{1,(0,1)} + y_{2,(1,0)}, \qquad \quad
    p_{(0,2)} = y_{2,(0,1)}.
  \end{equation*}
  For example, the linear form $p_{(1,1)}$ expresses that the monomial
  $x_{1}x_{2}$ appears from multiplication of $f_{1}$ with $x_{2}$ or
  $f_{2}$ with~$x_{1}$.  We now evaluate these linear forms on
  $g = 3x_{1}^{2} + 5x_{1}x_{2} + 7x_{2}^{2}$.  This polynomial can be
  expressed using the generators as $g = g_{1}f_{1} + g_{2}f_{2}$,
  where $g_{1}=3x_{1}+ax_{2}$ and $g_{2}=bx_{1}+7x_{2}$ with $a+b=5$.
  We have
  \begin{equation*}
    p_{(2,0)} (g_{1},g_{2}) = 3,\qquad\quad
    p_{(1,1)} (g_{1},g_{2}) = a + b = 5, \qquad\quad
    p_{(0,2)} (g_{1},g_{2}) = 7.
  \end{equation*}
\end{example}

\begin{example}\label{ex:principalId}
  Let $I=\< x_{1}^2+x_{1}x_{2}+x_{2}^2\>\subseteq K[x_{1},x_{2}]$.  In
  degree $4 = 2+2$ we have
  \begin{align*}
    p_{(4,0)} &= y_{(2,0)},  &
    p_{(3,1)} &= y_{(2,0)}+y_{(1,1)}, \\
    p_{(2,2)} &= y_{(2,0)}+y_{(1,1)}+y_{(0,2)}, & 
    p_{(1,3)} &= y_{(1,1)}+y_{(0,2)}, \\
    p_{(0,4)} &= y_{(0,2)}.
  \end{align*}
  By Lemma~\ref{lem:p-alpha}, every polynomial in $I^{(4)}$ has the
  form
  \begin{equation*}
    p_{(4,0)}x_{1}^4+p_{(3,1)}x_{1}^3x_{2}+p_{(2,2)}x_{1}^2x_{2}^2+p_{(1,3)}x_{1}x_{2}^3+p_{(0,4)}x_{2}^4.
  \end{equation*}
  If that polynomial is $g(x_{1}^{2}+x_{1}x_{2}+x_{2}^{2})$ with
  $g=g_{(2,0)}x_{1}^{2}+g_{(1,1)}x_{1}x_{2}+g_{(0,2)}x_{2}^{2}\in
  K[x_{1},x_{2}]^{(2)}$, then it equals
  \begin{equation*}
    g_{(2,0)}x_{1}^4+(g_{(2,0)}+g_{(1,1)})x_{1}^3x_2+(g_{(2,0)}+g_{(1,1)}+g_{(0,2)})x_{1}^2x_2^2
    +(g_{(1,1)}+g_{(0,2)})x_{1}x_2^3+g_{(0,2)}x_2^4.
  \end{equation*}
\end{example}

A first indication how this can yield shortness is the following
trivial observation: A monomial $x^{\alpha}$ does not appear in any
polynomial of an ideal if and only if the corresponding $p_{\alpha}$
is zero.
A more insightful approach uses Lemma~\ref{lem:p-alpha}: the existence
of a polynomial $f\in I^{(t+d)}$ with few terms is equivalent to the
existence of coefficients
$(g_1,\dots,g_r)\in\bigoplus_{i=1}^{r}K[x_1,\dots,x_n]^{(d)}$ such
that $p_{\alpha}(g_1,\dots,g_r)$ vanishes for many~$\alpha$.  In the
following lemma we dualize this to spans of the~$p_{\alpha}$.

\begin{lem}\label{lem:palphaDim} The vector space
  $I^{(t+d)}$ does not contain an $s$-short polynomial if and only if
  for all $S\subseteq M_{t+d}^n$ with $|S|=|M_{t+d}^n|-s$ it holds
  \begin{equation*}
    \Span\set{p_{\alpha}:\alpha\in S} = \Span\left\{p_{\alpha}:\alpha\in M_{t+d}^n\right\}.
  \end{equation*}
\end{lem}

\begin{proof}
  For each $\alpha\in M_{t+d}^n$ let $V_{\alpha}$ be the kernel of the
  $K$-linear map $p_{\alpha}$.  Consider the statement \emph{If
    $|M^{n}_{t+d}|-s$ terms of $f\in I^{(t+d)}$ vanish (meaning all
    but $s$ terms), then $f$ is the zero polynomial}.  This statement
  is equivalent to the equalities
  \begin{equation}\label{e:intersectKer}
    \bigcap_{\alpha \in S} V_{\alpha} = \bigcap_{\alpha\in
      M^{n}_{t+d}} V_{\alpha}\quad \text{for all}\quad S\subseteq
    M_{t+d}^n,\; |S|=|M_{t+d}^n|-s.
  \end{equation}
  The statement is also equivalent to $I^{(t+d)}$ not containing a
  nonzero polynomial with at most $s$ terms.  The statement of the
  lemma follows by applying vector space duality
  to~\eqref{e:intersectKer}.
\end{proof}

\begin{remark}\label{rem:hyperplanes}
  The \emph{hyperplanes} of a matroid are the codimension one flats,
  that is, the maximal subsets that do not span everything.  There is
  a representable matroid whose vectors are the
  $p_{\alpha}$ for all $\alpha\in M_{t+d}^{n}$.  The hyperplanes of that
  matroid are the maximal sets $S\subseteq M_{t+d}^{n}$ such that
  \[
    \Span\set{p_{\alpha}:\alpha\in S} \subsetneq
    \Span\left\{p_{\alpha}:\alpha\in M_{t+d}^n\right\}.
  \]
  By Lemma~\ref{lem:palphaDim}, the existence of a short polynomial is
  tied to the existence of a large hyperplane: an $s$-short polynomial
  exists if and only if a hyperplane of size at least
  $|M_{t+d}^{n}|-s$ exists.
\end{remark}

\begin{example}
  Applying Lemma~\ref{lem:palphaDim} to $I^{(4)}$ from
  Example~\ref{ex:principalId} we find
  \begin{equation*}
    \Span\set{p_{(4,0)},p_{(3,1)},p_{(2,2)},p_{(1,3)},p_{(0,4)}}=\Span\set{y_{(2,0)},y_{(1,1)},y_{(0,2)}}.
  \end{equation*}
  This vector space is $3$-dimensional. Does $I^{(4)}$ contain a
  binomial?  One quickly checks that
  $\dim\Span\set{p_{(4,0)},p_{(2,2)},p_{(1,3)}} = 2$ and in particular
  \begin{equation*}
    \Span\set{p_{\alpha}:\alpha\in S} \neq \Span\left\{p_{\alpha}:\alpha\in M_{t+d}^n\right\}.
  \end{equation*}
  Therefore with $s=2$ and $S:=\set{(4,0),(2,2),(1,3)}$ one has
  $|S|= 5-2 =|M_4^2|-s$.  By Lemma~\ref{lem:palphaDim}, the vector
  space $I^{(4)}$ contains a binomial.  And indeed, we find
  \begin{equation*}
    x^4-xy^3=(x^2-xy)(x^2+xy+y^2)\in I^{(4)}.
  \end{equation*}
  However, $I^{(4)}$ does not contain a monomial, because any four
  $p_{\alpha}$ span a $3$-dimensional space.
\end{example}

\begin{example}
  We view Example~\ref{ex:principalId} from the perspective of matroid
  theory.
  We order the $p_{\alpha}$ and the columns as
  $1\colon (4,0), 2\colon (3,1), \dots, 5\colon (0,4)$.  The rows are
  ordered as $1\colon (2,0), 2\colon (1,1), 3\colon (0,2)$.  The
  representable matroid of the $p_{\alpha}$ is then described by the
  matrix
  \[
    \begin{pmatrix}
      1 & 1 & 1 & 0 & 0 \\
      0 & 1 & 1 & 1 & 0 \\
      0 & 0 & 1 & 1 & 1 
    \end{pmatrix}.
  \]
  This matroid has the bases
  $ \set{123, 234, 124, 145, 245, 125, 135}, $ and the circuits
  $\set{134, 1245, 235}$.  The hyperplanes are
  $\set{12, 134, 15, 24, 235, 45}$.  The two ``large'' (3-element)
  hyperplanes indicate the presence of binomials in the ideal.
\end{example}

The following reformulation of Lemma~\ref{lem:palphaDim} turns out to
be useful.

\begin{lem}\label{lem:criterionTK}
  The vector space $I^{(t+d)}$ does not contain a nonzero $s$-short
  polynomial if and only if for every $S\subseteq M_{t+d}^n$ with
  $|S|=|M_{t+d}^n|-s$ and any $\beta\in M_{t+d}^{n}$, there exists a
  linear combination
  $p_{\beta}= \sum_{\alpha\in S}r_{\alpha}p_{\alpha}$,
  $r_{\alpha}\in K$.
\end{lem}

We close this section with some simple consequences.
\begin{prop}\label{prop:shortnessKnown}
  Let $I$ be a nonzero ideal.  Suppose that the shortest nonzero
  polynomial in $I^{(t+d)}$ has $s$ terms. Then
  \begin{enumerate}[(i)]
  \item\label{it:sK1}
    $\dim \Span\set{p_{\alpha}: \alpha\in M_{t+d}^n} \le
    |M_{t+d}^n|-s+1$,
  \item\label{it:sK2} For each $\gamma \in M^{d}_{n}$ let $n(\gamma)$
    be the number of $\alpha \in M^{t+d}_{n}$ such that
    $y_{i,\gamma} \in \supp (p_{\alpha})$.  That is,
    $n(\gamma) = |\set{\alpha: y_{i,\gamma} \in \supp
      (p_{\alpha})}|$.  Then
    \[
      s \ge \min_{\gamma\in M^{d}_{n} : n(\gamma) \neq 0} n(\gamma).
    \]
  \end{enumerate}
\end{prop}

\begin{proof}
  By Lemma~\ref{lem:palphaDim}, for any $S\subseteq M_{t+d}^n$ of
  cardinality $|M_{t+d}^n|-s+1$, the set
  $\set{p_{\alpha}\colon \alpha\in S}$ generates
  $\Span\set{p_{\alpha}: \alpha\in M_{t+d}^n}$, so that (\ref{it:sK1})
  follows.

  Since $I$ is not the zero ideal, there exists a variable
  $y_{i,\gamma}$ such that
  $\set{\alpha: y_{i,\gamma}\in\supp(p_{\alpha})}\neq\emptyset$.
  Therefore the minimum exists and is positive.  For a contradiction,
  assume that $n(\gamma)<s$ for a $\gamma$ that realizes the minimum.
  Fix $\beta \in M_{t+d}^{n}$ such that
  $y_{i,\gamma}\in\supp(p_{\beta})$.  Now choose a subset
  $S\subseteq M_{t+d}^n$ with $|S|=|M_{t+d}^n|-s+1$ and
  $\set{\alpha: y_{i,\gamma}\in\supp(p_{\alpha})}\cap S =
  \set{\beta}$.  This is possible since
  $|\set{\alpha:
    y_{i,\gamma}\in\supp(p_{\alpha})}\setminus\set{\beta}|\leq s-2$.
  By Lemma~\ref{lem:criterionTK} there exists a relation
  $p_{\beta} = \sum_{\alpha\in S}r_{\alpha}p_{\alpha}$.  This is a
  contradiction because $p_{\beta}$ contains $y_{i,\gamma}$ while
  $p_{\alpha}$ do not, when $\alpha\in S$.
\end{proof}


\section{Short polynomials in determinantal ideals}
\label{sec:shortPoly-in-detIdeals}

Let $X=(x_{ij})$ be an $m\times n$-matrix of indeterminates over $K$
and $K[X]$ the polynomial ring with indeterminates $X$ and
coefficients in the field~$K$.  For any $0< t \le m,n$, denote by
\( I_t = \< t\text{-minors of } X\> \) the \emph{determinantal ideal}
generated by the $t$-minors, i.e.\ the $t\times t$ subdeterminants
of~$X$.  We prove the following theorem.

\begin{thm}\label{thm:DeterAtMost}
  The shortness of $I_{t}$ is at least $\frac{t!}{2}+1$.
  That is, $I_{t}$ does not contain a nonzero
  polynomial with at most $\frac{t!}{2}$ terms.
\end{thm}

Theorem~\ref{thm:DeterAtMost} implies the theorem from the
introduction since the irreducible algebraic set of matrices of rank
at most $t-1$ is cut out by the prime ideal generated by all
$t$-minors.
Our proof strategy consists of explicitly describing the linear forms
$p_{\alpha}$.  This is possible by the combinatorial nature of
determinantal ideals.  To do so, we introduce some notation. Let
$I\subseteq[m]$ and $J\subseteq[n]$ be index sets of size~$t$, and
$
  S_{I,J} \defas \set{\sigma\colon I\to J\text{ bijective}}.
$
Elements of $S_{I,J}$ are \textit{permutations} and the
\textit{signum} of $\sigma\in S_{I,J}$ is
$
  \sgn(\sigma) \defas \sgn(\psi\circ\sigma\circ\phi)
$
where $\phi\colon [t]\to I$ and $\psi\colon J\to[t]$ are the unique
bijective and order preserving maps defined by~$\sigma$.  The
\textit{permutation matrix} $E_{\sigma}\in\set{0,1}^{m\times n}$ of
$\sigma$ has $(i,j)$-entry equal to $1$ if and only if $i\in I$ and
$\sigma(i)=j$.

Now fix a nonnegative integer (degree)~$d$.  In this setting the
linear forms~$p_{\alpha}$ use the variables $y_{(I,J),\gamma}$ where
$I\subseteq [m]$, $J\subseteq [n]$ with $\abs{I} = \abs{J}=t$ and
$\gamma \in M^{m\times n}_{d}$. 

\begin{lem}\label{lem:palpha-formula-It}
  Let $\alpha\in M_{t+d}^{m\times n}$.  For the ideal $I_{t}$, we have
  \begin{equation*}
    p_{\alpha}
     =\sum_{\substack{\sigma\in S_{I,J}\\E_{\sigma}\leq\alpha}}\sgn(\sigma)y_{(I,J),\alpha-E_{\sigma}},
  \end{equation*}
  where $I\subseteq[m]$ and $J\subseteq[n]$ both have
  cardinality $t$ and $E_{\sigma}\leq\alpha$ is defined entrywise.
\end{lem}

\begin{proof}
  We translate the definitions from
  Section~\ref{sec:short-and-linearforms} to this case.  Let
  $f_{(I,J)} = \det (x_{i,j})_{i\in I,j\in J}$.  This is a homogenous
  polynomial of degree $t$ with coefficients
  \begin{equation*}
    f_{(I,J),\beta}
    = \begin{cases} \sgn(\sigma), & \text{if $E_{\sigma}=\beta$ for some $\sigma\in S_{I,J}$}, \\
      0, & \text{otherwise.}\end{cases}
  \end{equation*}
  That is,
  $f_{(I,J)}=\sum_{\beta\in M_{t}^{m\times n}}f_{(I,J),\beta}x^{\beta}$.  Now
  apply Lemma~\ref{lem:palpha-formula} to the generators $f_{(I,J)}$
  of $I_{t}$ and find
  \begin{equation*}
    p_{\alpha}
    =\sum_{I,J}\sum_{\substack{\beta\in M_t^{m\times n},\gamma\in M_d^{m\times n}\\\beta+\gamma=\alpha}}f_{(I,J),\beta}y_{(I,J),\gamma}
    =\sum_{\substack{\sigma\in S_{I,J}\\E_{\sigma}\leq\alpha}}\sgn(\sigma)y_{(I,J),\alpha-E_{\sigma}}.  \qedhere
  \end{equation*}
\end{proof}

We want to apply Lemma~\ref{lem:criterionTK} after understanding the
linear relations between the~$p_{\alpha}$. First we examine the
procedure in an example.

\begin{example}\label{ex:some-relation}
  Suppose that $X$ is a $2\times 3$-matrix, and consider
  $I_{2}^{(3)}$, the vector space of homogeneous polynomials of
  degree~$3$ in the ideal $I_{2}$, generated by all $2$-minors of~$X$.
  That is, $d=1$, $m=t=2$, and $n=3$.  A nonzero linear form is
  \begin{equation*}
    p_{\left(\begin{smallmatrix}1&0&1\\0&1&0\end{smallmatrix}\right)}
    = y_{(\set{1,2},\set{1,2}),\left(\begin{smallmatrix}0&0&1\\0&0&0\end{smallmatrix}\right)}
    - y_{(\set{1,2},\set{2,3}),\left(\begin{smallmatrix}1&0&0\\0&0&0\end{smallmatrix}\right)}.
  \end{equation*}
  To find linear relations among the $p_{\alpha}$ we search for
  $p_{\alpha'}$ which also use the indeterminates
  $y_{(\set{1,2},\set{1,2}),\left(\begin{smallmatrix}0&0&1\\0&0&0\end{smallmatrix}\right)}$
  and
  $y_{(\set{1,2},\set{2,3}),\left(\begin{smallmatrix}1&0&0\\0&0&0\end{smallmatrix}\right)}$.
  As it turns out, the first is also found in the support of exactly
  one other linear form, namely
  $p_{\left(\begin{smallmatrix}0&1&1\\1&0&0\end{smallmatrix}\right)}$.
  The second indeterminate is also contained in the support of exactly
  one other linear form:
  $p_{\left(\begin{smallmatrix}1&1&0\\0&0&1\end{smallmatrix}\right)}$.
  And fortunately, these three make a nontrivial relation
  \begin{equation}\label{eq:example-rel}
    p_{\left(\begin{smallmatrix}1&0&1\\0&1&0\end{smallmatrix}\right)} +
    p_{\left(\begin{smallmatrix}1&1&0\\0&0&1\end{smallmatrix}\right)} +
    p_{\left(\begin{smallmatrix}0&1&1\\1&0&0\end{smallmatrix}\right)}=0.
  \end{equation}
  Moreover, the uniqueness of the latter two linear forms implies that
  this is (up to a scalar) the only relation containing
  $p_{\left(\begin{smallmatrix}1&0&1\\0&1&0\end{smallmatrix}\right)}$.
\end{example}

Our proof of the lower bound of the shortness of $I_{t}$ is a
generalization of the idea in the previous example.  To apply
Lemma~\ref{lem:criterionTK}, we pick some linear form $p_{\beta}$ and
determine a relation
$\sum_{\alpha\in S}r_{\alpha}p_{\alpha}=p_{\beta}$.  We proceed by
picking for each indeterminate in the support of $p_{\beta}$ a linear
form that eliminates this indeterminate in $p_{\beta}$.  In contrast
to the previous example, this does not immediately give a relation in
the general setting, so we iterate this step.

In the proof of Theorem~\ref{thm:DeterAtMost}, we can pick for each
indeterminate in the support of a $p_{\beta}$ a linear form in which
this indeterminate occurs with the other sign.  This is the crucial
technical observation used to establish the bound $\frac{t!}{2}+1$ on
the shortness of~$I_{t}$.

\begin{proof}[Proof of Theorem~\ref{thm:DeterAtMost}]
  By Lemma~\ref{lem:criterionTK}, we have to show that for every
  $d\geq 0$, every $S\subseteq M_{t+d}^{m\times n}$ with
  $|S|=|M_{t+d}^{m\times n}|-\frac{t!}{2}$, and every
  $\beta\in M_{t+d}^{m\times n}$, there exists a linear combination
  $\sum_{\alpha\in S}r_{\alpha}p_{\alpha}=p_{\beta}$.  So fix all
  those quantities.  If $\beta\in S$ the result follows.  Assume
  therefore that $\beta\not\in S$.

  Let $V_{0}\defas \set{\beta}$.  We iteratively define an increasing
  sequence
  \begin{equation*}
    V_{0}\subseteq V_{1} \subseteq V_{2} \subseteq \dots
  \end{equation*}
  of subsets of $S\cup\set{\beta}$.  Assume now that $V_{k}$ is
  defined for some $k\geq 0$.  Then for every
  $y_{(I,J),\alpha-E_{\sigma}}$ appearing in some $\supp (p_{\alpha})$
  with $\alpha \in V_{k}$ we pick one permutation
  $\pi_{\alpha,\sigma}$ according to the following rules:
  \begin{enumerate}
  \item If $\pi_{\alpha,\sigma}$ is already defined because
    $y_{(I,J),\alpha-E_{\sigma}}$ occured before, do nothing.
  \item If there exist a permutation
    $\tau\in S_{I,J}\setminus\set{\sigma}$ such that
    $\alpha-E_{\sigma}+E_{\tau}\in V_{k}$ and
    $\pi_{\alpha-E_{\sigma}+E_{\tau},\tau}=\sigma$, then set
    $\pi_{\alpha,\sigma}:=\tau$.
  \item\label{it:alternating} If not, choose $\pi_{\alpha,\sigma}\in S_{I,J}$ such that
    $\alpha-E_{\sigma}+E_{\pi_{\alpha,\sigma}}\in S$ and
    $\sgn(\sigma)=-\sgn(\pi_{\alpha,\sigma})$.
  \end{enumerate}
  The rules allow many different assignments of $\pi_{\alpha,\sigma}$
  and each suffices for the argument.  Picking $\pi_{\alpha,\sigma}$
  in step \ref{it:alternating} is possible because the cardinality of
  $S$ is equals that of the alternating subgroup of~$S_{I,J}$.  Using
  all the choices made, set
  \begin{equation*}
    V_{k+1}:=
    V_{k}\cup\set{\alpha-E_{\sigma}+E_{\pi_{\alpha,\sigma}} \colon
      \alpha\in V_{k} \text{ and } y_{(I,J),\alpha-E_{\sigma}}\in\supp(p_{\alpha})}
  \end{equation*}
as well as
\begin{equation*}
V:=\bigcup_{k=0}^{\infty}V_{k}.
\end{equation*}
We claim that $\sum_{\alpha\in V}p_{\alpha}=0$.  To prove this claim,
consider an arbitary indeterminate $y_{(I,J),\gamma}$ that appears in
the sum, i.e.\ such that the set
\[
  V_{y_{(I,J),\gamma}} \defas \set{\alpha\in V \colon
    y_{(I,J),\gamma}\in\supp(p_{\alpha})}
\]
is not empty.  The construction of $V$ shows that
$V_{y_{(I,J),\gamma}}$ is a disjoint union of subsets of the form
$ \set{\alpha, \alpha-E_{\sigma}+E_{\pi_{\alpha,\sigma}}}, $ where
$\sigma\in S_{I,J}$ satifies $\gamma=\alpha-E_{\sigma}$.  In
particular, we can pick pairwise distinct
$p_{\alpha_{1}},\dots,p_{\alpha_{l}}\in V_{y_{(I,J),\gamma}}$ and
permutations $\sigma_{1},\dots,\sigma_{l}\in S_{I,J}$ such that the
linear forms
$p_{\alpha_{1}-E_{\sigma_{1}}+E_{\pi_{\alpha_{1},\sigma_{1}}}},\dots,p_{\alpha_{l}-E_{\sigma_{l}}+E_{\pi_{\alpha_{l},\sigma_{l}}}}$
are pairwise distinct and $V_{y_{(I,J),\gamma}}$ is a disjoint union
\begin{equation*}
  V_{y_{(I,J),\gamma}}
   =\set{\alpha_{1},\dots,\alpha_{l}}
   \sqcup\set{\alpha_{1}-E_{\sigma_{1}}+E_{\pi_{\alpha_{1},\sigma_{1}}},\dotsc, \alpha_{l}-E_{\sigma_{l}}+E_{\pi_{\alpha_{l},\sigma_{l}}}}.
\end{equation*}
From this decomposition, it follows that the coefficient of
$y_{(I,J),\gamma}$ in 
$\sum_{\alpha\in V}p_{\alpha}$ equals
\begin{equation*}
  \sum_{i=1}^{l}\sgn(\sigma_{i})+\sgn(\pi_{\alpha_{i},\sigma_{i}})
  =\sum_{i=1}^{l}\sgn(\sigma_{i})-\sgn(\sigma_{i})
  =0.
\end{equation*}
Therefore the $y_{(I,J),\gamma}$-term and thus all terms in
$\sum_{\alpha\in V}p_{\alpha}$ vanish and
$\sum_{\alpha\in V}p_{\alpha}=0$ which proves the claim.  Now the
entire proof is finished since $\beta \in V$ and
$V\setminus \set{\beta} \subseteq S$.  Thus we get the required
expression
$ p_{\beta} = - \sum_{\substack{\alpha \in V\\ \alpha\neq \beta}}
p_{\alpha}.  $ for Lemma~\ref{lem:criterionTK}.
\end{proof}


\begin{remark}\label{r:breadt-first}
  In the proof of Theorem~\ref{thm:DeterAtMost}, let $G$ be the graph
  whose vertex set consists of all $\alpha$ such that $p_{\alpha}$ is
  not zero and edges between $p_{\alpha}$ and $p_{\alpha'}$ whenever
  those two linear forms share an indeterminate.  Starting from
  $p_{\beta}$ (or just any distinguished vertex) our proof first
  collects vertices adjacent to $p_{\beta}$ such that each
  indeterminate in $p_{\beta}$ is matched exactly once.  This step is
  repeated for every vertex until all of their indetermiantes are
  matched.  Consequently the proof implements a breadth-first search
  on this graph.  Figure~\ref{fig:graph_ex} contains the graph
  corresponding to degree three polynomials from
  Example~\ref{ex:some-relation}.  It seems plausible that more
  complicated but more efficient relations could be found by exploring
  a simplicial complex, so that in
  $p_{\beta} = -\sum_{\alpha} p_{\alpha}$ each indeterminate could
  appear more than once on each side of the equation.
\end{remark}

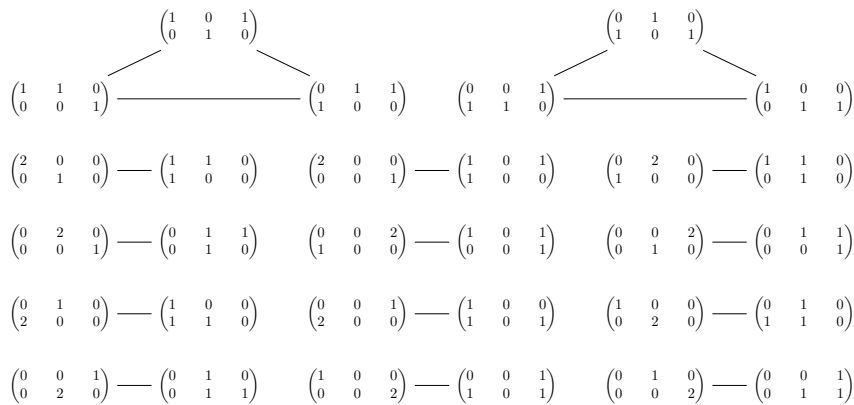
\begin{figure}
\begin{tikzpicture}[baseline= (a).base]
\node[scale=.55] (a) at (0,0){
  \begin{tikzcd}
  & \begin{pmatrix}1&0&1\\0&1&0\end{pmatrix} \arrow[dash]{dr} & & & \begin{pmatrix}0&1&0\\1&0&1\end{pmatrix} \arrow[dash]{dr} & \\
  \begin{pmatrix}1&1&0\\0&0&1\end{pmatrix}\arrow[dash]{ur} \arrow[dash]{rr} &    & \begin{pmatrix}0&1&1\\1&0&0\end{pmatrix} &
  \begin{pmatrix}0&0&1\\1&1&0\end{pmatrix}\arrow[dash]{ur} \arrow[dash]{rr} &    & \begin{pmatrix}1&0&0\\0&1&1\end{pmatrix} \\
  \begin{pmatrix}2&0&0\\0&1&0\end{pmatrix} \arrow[dash]{r} & \begin{pmatrix}1&1&0\\1&0&0\end{pmatrix} &
  \begin{pmatrix}2&0&0\\0&0&1\end{pmatrix} \arrow[dash]{r} & \begin{pmatrix}1&0&1\\1&0&0\end{pmatrix} &
  \begin{pmatrix}0&2&0\\1&0&0\end{pmatrix} \arrow[dash]{r} & \begin{pmatrix}1&1&0\\0&1&0\end{pmatrix} \\
  \begin{pmatrix}0&2&0\\0&0&1\end{pmatrix} \arrow[dash]{r} & \begin{pmatrix}0&1&1\\0&1&0\end{pmatrix} &
  \begin{pmatrix}0&0&2\\1&0&0\end{pmatrix} \arrow[dash]{r} & \begin{pmatrix}1&0&1\\0&0&1\end{pmatrix} &
  \begin{pmatrix}0&0&2\\0&1&0\end{pmatrix} \arrow[dash]{r} & \begin{pmatrix}0&1&1\\0&0&1\end{pmatrix} \\
  \begin{pmatrix}0&1&0\\2&0&0\end{pmatrix} \arrow[dash]{r} & \begin{pmatrix}1&0&0\\1&1&0\end{pmatrix} &
  \begin{pmatrix}0&0&1\\2&0&0\end{pmatrix} \arrow[dash]{r} & \begin{pmatrix}1&0&0\\1&0&1\end{pmatrix} &
  \begin{pmatrix}1&0&0\\0&2&0\end{pmatrix} \arrow[dash]{r} & \begin{pmatrix}0&1&0\\1&1&0\end{pmatrix} \\
  \begin{pmatrix}0&0&1\\0&2&0\end{pmatrix} \arrow[dash]{r} & \begin{pmatrix}0&1&0\\0&1&1\end{pmatrix} &
  \begin{pmatrix}1&0&0\\0&0&2\end{pmatrix} \arrow[dash]{r} & \begin{pmatrix}0&0&1\\1&0&1\end{pmatrix} &
  \begin{pmatrix}0&1&0\\0&0&2\end{pmatrix} \arrow[dash]{r} & \begin{pmatrix}0&0&1\\0&1&1\end{pmatrix} 
\end{tikzcd}
};
\end{tikzpicture}
\caption{The graph $G$ for $I_2^{(3)}$, where $X$ is a $2\times 3$-matrix.
  In this special case, every connected component gives rise to
  a linear relation. For instance, we find~\eqref{eq:example-rel} in
  its top left corner. \label{fig:graph_ex}}
\end{figure}

\begin{remark}\label{rem:signs}
  In the proof of Theorem~\ref{thm:DeterAtMost}, for each
  indeterminate in the support of a $p_{\beta}$ we pick a linear form
  in which this indeterminate occurs with the opposite sign.  The
  proof therefore constructs linear relations
  $p_{\beta}= \sum_{\alpha}r_{\alpha}p_{\alpha}$ in which all nonzero
  coefficients $r_{\alpha}$ equal $-1$.  This is a strong restriction,
  and in general there should exist more complicated relations giving
  better bounds.  In particular, describing all relations would yield
  the exact bound.
\end{remark}

Since the generators of $I_{t}$ are $t!$-short, we state the following
\begin{conj}\label{conj:It-t!short}
$I_{t}$ is $t!$-short.
\end{conj}

When Conjecture~\ref{conj:It-t!short} is resolved in one way or the
other, it would be interesting to compare with permanental ideals,
which in many ways are more complicated than determinantal ideals.

\medskip
 
\bibliographystyle{amsplain}
\bibliography{shortpolynomials}

\bigskip \medskip

\noindent
\footnotesize {\bf Authors' addresses:}

\smallskip

\noindent Thomas Kahle, OvGU Magdeburg, Germany,
{\tt thomas.kahle@ovgu.de}

\noindent Finn Wiersig, University of Oxford, UK,
{\tt finn.wiersig@maths.ox.ac.uk}

\end{document}